\documentclass[a4paper,11pt]{article}

\usepackage[applemac]{inputenc}
\usepackage{enumerate}
\usepackage{lmodern}
\usepackage[T1]{fontenc}
\usepackage{verbatim}
\usepackage{textcomp}
\usepackage[english]{babel}
\usepackage[a4paper,vmargin={3.5cm,3.5cm},hmargin={2.5cm,2.5cm}]{geometry}
\usepackage[font=sf, labelfont={sf,bf}, margin=1cm]{caption}
\usepackage[pdftex]{hyperref}
\usepackage[pdftex]{color,graphicx}

\usepackage{amsmath,amsfonts,amssymb,amsthm,mathrsfs}

\usepackage{colonequals} 

\newtheorem{theorem}{Theorem}[]
\newtheorem{definition}{Definition}[]
\newtheorem{proposition}[theorem]{Proposition}
\newtheorem{lemma}[theorem]{Lemma}

\theoremstyle{definition}

\def\cc{{\mathcal C}}

\def\t{{\mathcal T}}

\def\ve{\varepsilon}

\def\T{{\mathbb T}}
\def\R{{\mathbb R}}

\def\N{{\mathbb N}}

\def\build#1_#2^#3{\mathrel{\mathop{\kern 0pt#1}\limits_{#2}^{#3}}}

\def\ind{{\bf 1}_}

\def\AGW{$\mathbf{AGW}_\lambda$ }
\def\reroot{\mathsf{ReRoot}}
\def\join{\!\!-\!\!\!\bullet}

\title{Harmonic measure for biased random walk in a supercritical Galton--Watson tree}

\author{Shen LIN 
\thanks{Supported in part by the grant ANR-14-CE25-0014 (ANR GRAAL)} \\
\small \it Sorbonne Universit\'e, Laboratoire de Probabilit\'es, Statistique et Mod\'elisation, Paris, France \\
\small \textit{E-mail}: \texttt{shen.lin.math@gmail.com} }

\date{\tiny\today}

\begin{document}

\maketitle

\begin{abstract}
We consider random walks $\lambda$-biased towards the root on a Galton--Watson tree, whose offspring distribution $(p_k)_{k\geq 1}$ is non-degenerate and has finite mean $m>1$. In the transient regime $0<\lambda <m$, the loop-erased trajectory of the biased random walk defines the $\lambda$-harmonic ray, whose law is the $\lambda$-harmonic measure on the boundary of the Galton--Watson tree. We answer a question of Lyons, Pemantle and Peres \cite{LPP97} by showing that the $\lambda$-harmonic measure has a.s.~strictly larger Hausdorff dimension than the visibility measure, which is the harmonic measure corresponding to the simple forward random walk. We also prove that the average number of children of the vertices along the $\lambda$-harmonic ray is a.s.~bounded below by $m$ and bounded above by $m^{-1}\sum k^2 p_k$. Moreover, at least for $0<\lambda \leq 1$, the average number of children of the vertices along the $\lambda$-harmonic ray is a.s.~strictly larger than that of the $\lambda$-biased random walk trajectory. We observe that the latter is not monotone in the bias parameter~$\lambda$.


\smallskip
\noindent{\bf AMS 2010 Classification Numbers.} 60J15, 60J80. 
\end{abstract}

\section{Introduction}

Consider a Galton--Watson tree $\T$ rooted at $e$ with a non-degenerate offspring distribution $(p_k)_{k\geq 0}$. We suppose that $p_0=0$, $p_k<1$ for all $k\geq 1$, and the mean offspring number $m=\sum_{k\geq 1} k p_k \in (1,\infty)$. So the Galton--Watson tree $\T$ is supercritical and leafless. 
Let $\t$ be the space of all infinite rooted trees with no leaves. 
The law of $\T$ is called the Galton--Watson measure $\mathbf{GW}$ on $\t$.
For every vertex $x$ in $\T$, let $\nu(x)$ stand for its number of children. We denote by $x_*$ the parent of $x$ and by $xi, 1\leq i\leq \nu(x)$, the children of $x$.

For $\lambda\geq 0$, conditionally on $\T$, the $\lambda$-biased random walk $(X_n)_{n\geq 0}$ on $\T$ is a Markov chain starting from the root $e$, such that, from the vertex $e$ all transitions to its children are equally likely, whereas for every vertex $x\in \T$ different from $e$,  
\begin{align*}
P_\T(X_{n+1}=x_* \mid X_n=x) &= \frac{\lambda}{\nu(x)+\lambda},\\
P_\T(X_{n+1}=xi \mid X_n=x) &= \frac{1}{\nu(x)+\lambda}, \quad \mbox{ for every } 1\leq i\leq \nu(x).
\end{align*}
Note that $\lambda=1$ corresponds to the simple random walk on $\T$, and $\lambda=0$ corresponds to the simple \emph{forward} random walk with no backtracking.
Lyons established in~\cite{Ly90} that $(X_n)_{n\geq 0}$ is almost surely transient if and only if $\lambda<m$. Throughout this work, we assume $\lambda <m$ and hence the $\lambda$-biased random walk is always transient. 

For a vertex $x\in \T$, let $|x|$ stand for the graph distance from the root $e$ to $x$. 
Let $\partial \T$ denote the boundary of $\T$, which is defined as the set of infinite rays in $\T$ emanating from the root. 
Since $(X_n)_{n\geq 0}$ is transient, its loop-erased trajectory defines a unique infinite ray $\Xi_\lambda \in \partial \T$, whose distribution is called the $\lambda$-harmonic measure. 
We call $\Xi_\lambda$ the $\lambda$-harmonic ray in $\T$.

For different rays $\xi, \eta\in \partial \T$, let $\xi\wedge \eta$ denote the vertex common to both $\xi$ and $\eta$ that is farthest from the root. We define the metric 
\begin{equation*}
d(\xi,\eta)\colonequals \exp(-|\xi\wedge \eta|) \mbox{ for } \xi, \eta\in \partial \T, \xi\neq \eta.
\end{equation*}
Under this metric, the boundary $\partial \T$ has a.s.~Hausdorff dimension $\log m$. 
Lyons, Pemantle and Peres \cite{LPP95,LPP96} showed the dimension drop of harmonic measure: for all $0\leq \lambda<m$, the Hausdorff dimension of the $\lambda$-harmonic measure is a.s.~a constant $d_\lambda< \log m$.
The 0-harmonic measure associated with the simple forward random walk was called \emph{visibility measure} in \cite{LPP95}. Its Hausdorff dimension is a.s.~equal to the constant $\sum_{k\geq 1} (\log k)p_k =\mathbf{GW}[\log \nu]$, where we write $\nu=\nu(e)$ for the offspring number of the root under $\mathbf{GW}$.

Recently, Berestycki, Lubetzky, Peres and Sly~\cite{BLPS} applied the dimension drop result $d_1<\log m$ to show cutoff for the mixing time of simple random walk on a random graph starting from a typical vertex. 
The Hausdorff dimension of the 0-harmonic measure was similarly used in~\cite{BLPS} and independently used by Ben-Hamou and Salez in~\cite{BHS} to determine the mixing time of the non-backtracking random walk on a random graph. 

The primary result of this work answers a question of Ledrappier posed in~\cite{LPP97}. This question is also stated as Question 17.28 in Lyons and Peres' book \cite{LP-book}. 

\begin{theorem}
\label{thm:dim-harm}
For all $\lambda\in(0,m)$, we have $d_\lambda > \mathbf{GW}[\log \nu]$, meaning that the Hausdorff dimension of the $\lambda$-harmonic measure is a.s.~strictly larger than the Hausdorff dimension of the 0-harmonic measure. Moreover, 
\begin{equation*}
\lim_{\lambda \to 0^+} d_\lambda = \mathbf{GW}[\log \nu] \quad \mbox{ and } \quad \lim_{\lambda \to m^-} d_\lambda = \log m.
\end{equation*}
\end{theorem}

When $\lambda$ increases to the critical value $m$, it is non-trivial that the support of the $\lambda$-harmonic measure has its Hausdorff dimension tending to that of the whole boundary. 
Besides, Jensen's inequality implies $\mathbf{GW}[\log \nu]> -\log\mathbf{GW}[\nu^{-1}]$. The preceding theorem thus improves the lower bound $d_\lambda>-\log\mathbf{GW}[\nu^{-1}] $ shown by Vir\'ag in Corollary 7.2 of~\cite{V2000}.

The proof of Theorem~\ref{thm:dim-harm} originates from the construction of a probability measure $\mu_{\mathsf{HARM}_\lambda}$ on~$\t$ that is stationary and ergodic for the harmonic flow rule. In Section \ref{sec:harm-invar} below, its Radon--Nikod\'ym derivative with respect to $\mathbf{GW}$ is given by \eqref{eq:density-Harm}. 
Note that an equivalent formula is also obtained independently by Rousselin \cite{Rou}.
We derive afterwards an explicit expression for the dimension $d_\lambda$, and prove Theorem~\ref{thm:dim-harm} in Section~\ref{sec:dim-harm}.
Our way to find the harmonic-stationary measure $\mu_{\mathsf{HARM}_\lambda}$ is inspired by a recent work of A\"id\'ekon~\cite{Aid}, in which he found the explicit stationary measure of the environment seen from a $\lambda$-biased random walk. It renders possible an application of the ergodic theory on Galton--Watson trees developed in \cite{LPP95} to the biased random walk. After introducing the escape probability of $\lambda$-biased random walk on a tree in Section~\ref{sec:esc}, we will give a precise description of A\"id\'ekon's stationary measure in Section~\ref{sec:agw}.

\smallskip
Apart from the Hausdorff dimension of harmonic measure, another quantity of interest is the average number of children of vertices visited by the harmonic ray $\Xi_\lambda$ or the $\lambda$-biased random walk $(X_n)_{n\geq 0}$ on $\T$. For an infinite path $\overset \rightarrow x=(x_k)_{k\geq 0}$ in $\T$, if the limit
\begin{equation*}
\lim_{n\to \infty} \frac{1}{n}\sum_{k=0}^n \nu(x_k)
\end{equation*}
exists, we call it the average number of children of the vertices along the path $\overset \rightarrow x$.
Section~\ref{sec:average-nb-child} will be devoted to comparing the average number of children of vertices along different random paths in $\T$. The main results in this direction are summarized in the following way.

\begin{theorem}
\label{thm:nb-children}

\begin{enumerate}
\item[(i)] For all $\lambda\in(0,m)$, the average number of children of the vertices along the $\lambda$-harmonic ray $\Xi_\lambda$ is a.s.~strictly larger than $m$, and strictly smaller than $m^{-1} \sum k^2 p_k$;
\item[(ii)] The average number of children of the vertices along the $\lambda$-biased random walk $(X_n)_{n\geq 0}$ is a.s.~strictly smaller than $m$ when $\lambda \in(0,1)$, equal to $m$ when $\lambda=0$ or $1$, and strictly larger than $m$ when $\lambda\in (1,m)$;
\item[(iii)] For $\lambda \in (0,1]$, the average number of children of the vertices along the $\lambda$-harmonic ray $\Xi_\lambda$ is a.s.~strictly larger than the average number of children of the vertices along the $\lambda$-biased random walk $(X_n)_{n\geq 0}$. 
\end{enumerate}
\end{theorem}

Assertion (iii) above is a direct consequence of assertions (i) and (ii). We conjecture that the same result holds for all $\lambda\in(0,m)$, not merely for $\lambda \in (0,1]$.

Assertion (i) in Theorem \ref{thm:nb-children} was first suggested by some numerical calculations in the case $\lambda=1$ mentioned at the end of Section 17.10 in \cite{LP-book}. By the strong law of large numbers, the average number of children seen by the simple forward random walk is a.s.~equal to $m$. 
On the other hand, the uniform measure on the boundary of $\T$ can be defined by putting mass~1 uniformly on the vertices of level $n$ in $\T$ and taking the weak limit as $n\to \infty$. We say that a random ray in $\T$ is uniform if it is distributed according to the uniform measure on $\partial \T$. 
When $\sum (k\log k)p_k <\infty$, the uniform measure on $\partial \T$ has a.s.~Hausdorff dimension $\log m$, and the uniform ray in $\T$ can be identified with the distinguished infinite path in a size-biased Galton--Watson tree. In particular, the average number of children seen by the uniform ray in $\T$ is equal to $m^{-1} \sum k^2 p_k$. For more details we refer the reader to Section 6 of \cite{LPP95} or Chapter 17 of \cite{LP-book}.
 
The FKG inequality for product measures (also known as the Harris inequality) turns out to be extremely useful in proving Theorem \ref{thm:nb-children}. In Section~\ref{sec:average-nb-child}, assertion (i) will be derived from Propositions \ref{prop:child-harm-visi} and \ref{prop:child-harm-unif}, while assertion (ii) will be shown as Proposition~\ref{prop:child-rw-path}. 

It is worth pointing out that the average number of children seen by the $\lambda$-biased random walk is \emph{not} monotone with respect to $\lambda$, because its right continuity at $0$ (established in Proposition~\ref{prop:child-rw-path-limit-0}), together with assertion (ii) in Theorem \ref{thm:nb-children}, implies that the average number of children seen by the $\lambda$-biased random walk cannot be monotonic nondecreasing for all $\lambda\in(0,1)$. 
This lack of monotonicity might be explained by two opposing effects of having a small bias $\lambda$: on the one hand, it helps the random walk to escape faster to infinity, and a high-degree path is in favour of the escape of the $\lambda$-biased random walk, but on the other hand, small bias implies less backtracking, so the $\lambda$-biased random walk spends less time on high-degree vertices.

We close this introduction by mentioning that the following question from \cite{LPP97} remains open.

\noindent{\bf Question 1.} \textit{Is the dimension $d_\lambda$ of the $\lambda$-harmonic measure nondecreasing for $\lambda\in(0,m)$?}

Taking into account the previous discussion, we find it intriguing to ask a similar question:

\noindent{\bf Question 2.} \textit{Is the average number of children of the vertices along the $\lambda$-harmonic ray in~$\T$ nondecreasing for $\lambda\in(0,m)$? Does the same monotonicity holds for the average number of children of the vertices along the $\lambda$-biased random walk, when $\lambda \in [1,m)$?}

\section{Escape probability and the effective conductance}
\label{sec:esc}

For a tree $T\in \t$ rooted at $e$, we define $T_*$ as the tree obtained by adding to $e$ an extra adjacent vertex $e_*$, called the parent of $e$. The new tree $T_*$ is naturally rooted at $e_*$. 
For a vertex $u\in T$, the descendant tree $T_u$ of $u$ is the subtree of $T$ formed by those edges and vertices which become disconnected from the root of $T$ when $u$ is removed. By definition, $T_u$ is rooted at $u$. 

Unless otherwise stated, we assume $\lambda \in (0,m)$ in the rest of the paper.
Under the probability measure $P_T$, let $(X_n)_{n \geq 0}$ denote a $\lambda$-biased random walk on $T_*$. 
For any vertex $u\in T$, define $\tau_u\colonequals \min\{n\geq 0 \colon X_n=u\}$ the hitting time of~$u$, with the usual convention that $\min \emptyset=\infty$.
Let 
\begin{equation*}
\beta_\lambda(T)\colonequals P_T(\tau_{e_*}=\infty\mid X_0=e)= P_T(\forall n\geq 1, X_n\neq e_* \mid X_0=e)
\end{equation*}
be the probability of never visiting the parent $e_*$ of $e$ when starting from~$e$. 
For notational ease, we will make implicit the dependency in $\lambda$ of the escape probability by writing $\beta(T)=\beta_\lambda(T)$.
For $\mathbf{GW}$-a.e.~$T$, $0<\beta(T)<1$.
By coupling with a biased random walk on $\N$, we see that $\beta(T)>1-\lambda$.
Moreover, Lemma 4.2 of \cite{Aid} shows that 
\begin{equation}
\label{eq:Aid-lemma}
0< \mathbf{GW}\bigg[\frac{1}{\lambda-1+\beta(T)}\bigg]<\infty. 
\end{equation}
For a vertex $u\in T$, $|u|=1$, the probability that a $\lambda$-harmonic ray in $T$ passes through $u$ is 
\begin{equation*}
\frac{\beta(T_u)}{\sum_{|w|=1}\beta(T_w)}.
\end{equation*}

If the tree $T$ is viewed as an electric network, and if the conductance of an edge linking
vertices of level $n$ and $n+1$ is $\lambda^{-n}$, then $\cc_\lambda(T)$ denotes the effective conductance of $T$ from its root to infinity. 
As for the escape probability, we will write $\cc(T)$ for $\cc_\lambda(T)$ to simplify the notation.
Using the link between reversible Markov chains and electric networks, we know that 
\begin{equation}
\label{eq:conductance}
\beta(T)=\cc(T_*)=\frac{\cc(T)}{\lambda +\cc(T)} \quad \mbox{ and }\quad  \cc(T)= \frac{\lambda \beta(T)}{1- \beta(T)}.
\end{equation} 
This relationship between $\beta(T)$ and $\cc(T)$ will be used repeatedly. 
Since $\cc(T)>\beta(T)$, the lower bound $\cc(T)>1-\lambda$ also holds. 
Moreover, for all $x\in \R$, 
\begin{equation*}
\frac{\lambda^{-1} x \, \cc(T)}{(\lambda-1+\cc(T))(1+\lambda^{-1} x)+\lambda^{-1} x}=
\frac{\beta(T)x}{\lambda-1+\beta(T)+x}.
\end{equation*}
Taking $x=\cc(T')$ for another tree $T'$ yields the following identity
\begin{equation}
\label{eq:cond-symmetry}
\frac{\beta(T')\cc(T)}{\lambda-1+\beta(T')+\cc(T)}=\frac{\beta(T)\cc(T')}{\lambda-1+\beta(T)+\cc(T')}.
\end{equation}
Using \eqref{eq:conductance} we can also verify that
\begin{align}
\label{eq:cond-identity}
(\lambda-1+\beta(T)+\cc(T'))(1+\lambda^{-1} \cc(T)) & =\lambda(1+\lambda^{-1} \cc(T))(1+\lambda^{-1} \cc(T'))-1 \nonumber \\
& = (\lambda-1+\beta(T')+\cc(T))(1+\lambda^{-1} \cc(T')). 
\end{align}

The following integrability result will be used to prove the inequality $d_\lambda>\mathbf{GW}[\log \nu]$. 

\begin{lemma}
\label{lem:mean-resistance-log}
For $0<\lambda<m$, we have $\mathbf{GW}\big[\log \frac{1}{\beta(T)}\big]<\infty$. 
\end{lemma}

\begin{proof}
Let $T_1,\ldots, T_\nu$ be the descendant trees of the children of the root in $T$. By the parallel law of conductances, $\beta(T)=\sum_{i=1}^\nu \beta(T_i)$. Recall that 
\begin{equation*}
\frac{1}{\beta(T)}=\frac{\cc(T)+\lambda}{\cc(T)} = 1+\frac{\lambda}{\sum_{i=1}^\nu \beta(T_i)}.
\end{equation*}
Taking $x=\lambda (\sum_{i=1}^{\nu} \beta(T_i))^{-1}$ and $x_0=\lambda \nu^{-1}\leq x$ in the inequality $\log(1+x)\leq \log x +\log (1+x_0^{-1})$,
we deduce that 
\begin{equation*}
\log \frac{1}{\beta(T)} \leq \log (1+\lambda^{-1} \nu) +\log \lambda +\log \frac{1}{\sum_{i=1}^{\nu} \beta(T_i)}.
\end{equation*}
Let $\ve<1$ be some positive number. Then, 
\begin{equation*}
\log \frac{1}{\sum_{i=1}^{\nu} \beta(T_i)} \leq \bigg(\log \frac{1}{\beta(T_1)}\bigg)\ind{\{\beta(T_i)\leq \ve, \forall i\geq 2\}}+\log\ve^{-1}.
\end{equation*}
By convention, the indicator function above is equal to 1 over the event $\{\nu=1\}$.
Taking expectation gives 
\begin{equation*}
\mathbf{GW}\bigg[n\wedge \log \frac{1}{\beta(T)} \bigg] \leq \mathbf{GW}\big[\log (\lambda+\nu)\big] +\log \ve^{-1}+ \mathbf{GW}\bigg[n \wedge \log \frac{1}{\beta(T)} \bigg] \mathbf{GW}\big[q_\ve ^{\nu-1}\big],
\end{equation*}
where $q_\ve\colonequals \mathbf{GW}(\beta(T)\leq \ve)$. Since $q_\ve \to 0$ when $\ve \to 0$, we can take $\ve$ small enough such that 
\begin{equation*}
A_\ve \colonequals \mathbf{GW}[q_\ve ^{\nu-1}]<1. 
\end{equation*}
Hence, we obtain 
\begin{equation*}
\mathbf{GW}\bigg[n\wedge \log \frac{1}{\beta(T)} \bigg] \leq \frac{\mathbf{GW}\big[\log (\lambda+\nu)\big] +\log \ve^{-1}}{1-A_\ve}.
\end{equation*}
Taking the limit $n\to \infty$ finishes the proof.
\end{proof}

\section{Stationary measure of the tree seen from random walk}
\label{sec:agw}

We set up some notation before presenting A\"id\'ekon's stationary measure.
For a~rooted tree $T\in \t$, its boundary $\partial T$ is the set of all rays starting from the root. Clearly, one can identify $\partial T_*$ with $\partial T$.
Let 
\begin{equation*}
\t^* \colonequals \{(T, \xi)\mid T\in \t, \xi=(\xi_n)_{n\geq 0} \in \partial T \}
\end{equation*}
denote the space of trees with a marked ray. By definition, $\xi_0$ coincides with the root vertex of~$T$.
If $T_1$ and $T_2$ are two trees rooted respectively at $e_1$ and $e_2$, we define $T_1 \join T_2$ as the tree rooted at the root $e_2$ of $T_2$ formed by joining the roots of $T_1$ and $T_2$ by an edge. The root $e_2$ is the parent of $e_1$ in $T_1 \join T_2$, thus we will not distinguish $e_2$ from $(e_1)_*$.
Given a ray $\xi\in \partial T$, there is a unique tree $T^+$ such that $T =T_{\xi_1} \join T^+$. Therefore, $\t^*$ is in bijection with the space
\begin{equation*}
\big \{(T \join T^+, \xi)\mid T, T^+ \in \t, \xi=(\xi_n)_{n\geq 0} \in \partial T \big\}.
\end{equation*}

Introducing a marked ray helps us to keep track of the past trajectory of the biased random walk. In particular, the initial starting point of the random walk, towards which the bias is exerted, would be represented by the marked ray at infinity. 
To be more precise, if we assign a vertex $u\in T$ to be the new root of the tree $T$, the re-rooted tree will be written as $\reroot(T,u)$. Given $\xi=(\xi_n)_{n\geq 0} \in \partial T$, we say that $x$ is the $\xi$-parent of $y$ in $T$ if $x$ becomes the parent of $y$ in the tree $\reroot(T,\xi_n)$ for all sufficiently large $n$. A random walk on $T$ is \emph{$\lambda$-biased towards $\xi$} if the random walk always moves to its $\xi$-parent with probability $\lambda$ times that of moving to one of the other neighbors. 

We consider the Markov chain on $\t^*$ that, starting from some fixed tree $T$ with a marked ray $\xi=(\xi_n)_{n\geq 0}$, is isomorphic to a random walk on $T$ $\lambda$-biased towards $\xi$. Recall that $\nu(\xi_0)$ is the number of edges incident to the root. 
The transition probabilities $\mathbf{p}_{\mathsf{RW}_\lambda}$ of this Markov chain are defined as follows: 
\begin{itemize}
\item If $T'=\reroot(T, x)$ and $\xi'=(x,\xi_0,\xi_1, \xi_2, \ldots)$ with a vertex $x$ adjacent to $\xi_0$ being different from $\xi_1$, 
\[
\mathbf{p}_{\mathsf{RW}_\lambda}((T,\xi), (T',\xi'))=\frac{1}{\nu(\xi_0)-1+\lambda};
\]
\item If $T'=\reroot(T, \xi_1)$ and $\xi'=(\xi_1,\xi_2,\ldots)$, 
\[
\mathbf{p}_{\mathsf{RW}_\lambda}((T,\xi), (T',\xi'))=\frac{\lambda}{\nu(\xi_0)-1+\lambda};
\]
\item Otherwise, $\mathbf{p}_{\mathsf{RW}_\lambda}((T,\xi), (T',\xi'))=0$.
\end{itemize}

We proceed to define the environment measure that is invariant under re-rooting along a $\lambda$-biased random walk. 
Let $\T$ and $\T^+$ be two independent Galton--Watson trees of offspring distribution $(p_k)_{k\geq 0}$.  We write $e$ for the root vertex of $\T$, and $e^+$ for the root vertex of $\T^+$. Let $\nu^+$ denote the number of children of $e^+$ in $\T^+$. Similarly, let $\nu$ denote the number of children of $e$ in $\T$.  Note that the number of children of $e^+$ in $\T \join \T^+$ is $\nu^+ +1$.
Conditionally on $(\T, \T^+)$, let $\mathcal{R}$ be a random ray in $\T$ distributed according to the $\lambda$-harmonic measure on $\partial \T$. We assume that $(\T \join \T^+, \mathcal{R})$ is defined under the probability measure $P$.

\begin{figure}[ht]
	\begin{center}
	\includegraphics[width=11cm]{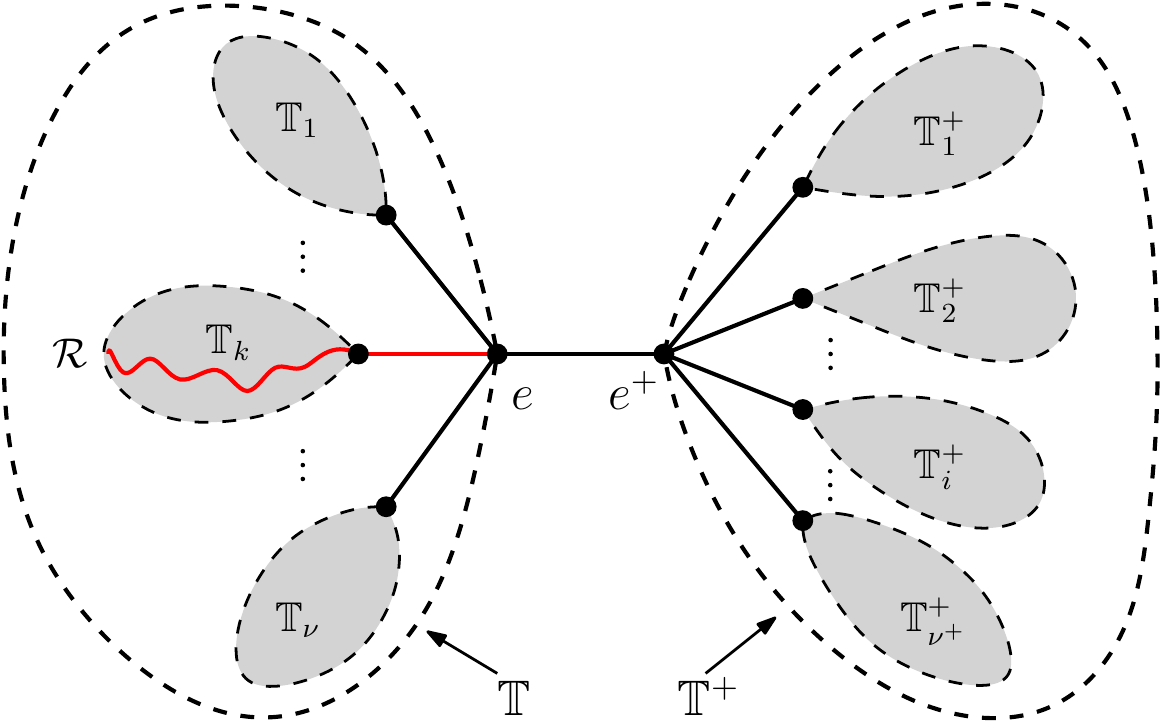}
	\caption{
		\label{fig:agw} 
		The random tree $\T \join \T^+$ rooted at $e^+$ with a marked ray $\mathcal{R}$
		}
	\end{center}
\end{figure}

\begin{definition}
The $\lambda$-augmented Galton--Watson measure \AGW is defined as the probability measure on $\t^*$ that is absolutely continuous with respect to the law of $(\T \join \T^+, \mathcal{R})$ with density 
\begin{equation}
\label{eq:density}
c_\lambda ^{-1} \frac{(\lambda+\nu^+) \beta(\T)}{\lambda-1+\beta(\T)+\cc(\T^+)},
\end{equation}
where 
\begin{equation*}
c_\lambda= E \bigg[ \frac{(\lambda+\nu^+) \beta(\T)}{\lambda-1+\beta(\T)+\cc(\T^+)}\bigg]
\end{equation*}
is the normalizing constant. 
\end{definition}

It follows from the inequality $\lambda-1+\cc(\T^+)>0$ that
\begin{equation*}
c_\lambda=E \bigg[ \frac{(\lambda+\nu^+) \beta(\T)}{\lambda-1+\beta(\T)+\cc(\T^+)}\bigg] < E\big[\lambda+\nu^+ \big]=\lambda+m.
\end{equation*}

Let $\T^+_1,\ldots, \T^+_{\nu^+}$ denote the descendant trees of the children of $e^+$ in $\T^+$. With a slight abuse of notation, let $\T_1,\ldots, \T_{\nu}$ denote the descendant trees of the children of $e$ inside $\T$. See Fig.~\ref{fig:agw} for a schematic illustration.
By the parallel law of conductances, 
\begin{equation}
\label{eq:C-Bsum}
\cc(\T^+)= \sum_{i=1}^{\nu^+} \beta(\T^+_i) \quad \mbox{and} \quad 
\cc(\T)= \sum_{i=1}^{\nu} \beta(\T_i).
\end{equation}
We will frequently use the branching property that conditionally on $\nu^+$, the collection of trees $\{\T,\T^+_1,\ldots, \T^+_{\nu^+} \}$ are independent and identically distributed according to $\mathbf{GW}$. 
 
According to Theorem 4.1 in \cite{Aid}, the $\lambda$-augmented Galton--Watson measure \AGW is the asymptotic distribution of the environment seen from the $\lambda$-biased random walk on $\T$. 

\begin{proposition}
\label{prop:invariance}
The Markov chain with transition probabilities $\mathsf{p}_{\mathsf{RW}_\lambda}$ and initial distribution \AGW is stationary. 
\end{proposition}

\begin{proof}
Let $F\colon \t \to \R^+$ and $G\colon \t^* \to \R^+$ be nonnegative measurable functions. Let $(\tilde \T \join \tilde \T^+, \tilde{\mathcal{R}})$ denote the tree with a marked ray obtained from $(\T \join \T^+, \mathcal{R})$ by performing a one-step transition according to $\mathbf{p}_{\mathsf{RW}_\lambda}$. It suffices to show that 
\begin{equation*}
E\Big[\frac{(\lambda+\nu^+) \beta(\T)}{\lambda-1+\beta(\T)+\cc(\T^+)} F(\tilde \T^+)G(\tilde \T, \tilde{\mathcal{R}})\Big] =
E\Big[\frac{(\lambda+\nu^+) \beta(\T)}{\lambda-1+\beta(\T)+\cc(\T^+)} F(\T^+)G(\T, \mathcal{R})\Big].
\end{equation*}
To compute the left-hand side, we need to distinguish two different situations. 

{\bf Case I}: There exists $1\leq i\leq \nu^+$ such that the root of $\T^+_i$ becomes the new root of $\tilde \T \join \tilde \T^+$. For each $i\in [1,\nu^+]$, it happens with probability $1/(\nu^+ +\lambda)$. In this case, 
\begin{equation*}
\tilde \T^+=\T^+_i \quad \mbox{and}\quad \tilde \T=\T\join \T_{\neq i}^+,
\end{equation*}
where $\T_{\neq i}^+$ stands for the tree rooted at $e^+$ containing only the descendant trees $\{\T^+_j, 1\leq j \leq \nu^+, j\neq i\}$ together with the edges connecting their roots to $e^+$. It is easy to see that $\T^+_i$ and $\T\join \T_{\neq i}^+$ are two i.i.d.~Galton--Watson trees. Meanwhile, $\tilde{\mathcal{R}}\in \partial \tilde \T$ is the ray $\mathcal{R}^+$ obtained by adding the vertex $e^+$ to the beginning of the sequence $\mathcal{R}$. We set accordingly 
\begin{align*}
I & \colonequals E\bigg[ \frac{(\lambda+\nu^+) \beta(\T)}{\lambda-1+\beta(\T)+\cc(\T^+)} \sum_{i=1}^{\nu^+} \frac{1}{\nu^+ +\lambda } F(\T^+_i)G(\T \join \T^+_{\neq i}, \mathcal{R}^+)\bigg] \\
&= E\bigg[ \frac{\beta(\T)}{\lambda-1+\beta(\T)+\cc(\T^+)} \sum_{i=1}^{\nu^+} F(\T^+_i)G(\T \join \T^+_{\neq i}, \mathcal{R}^+)\bigg].
\end{align*}
Given $\T$ and $\T^+$, we let $\mathcal{R}_{\neq i}$ be a random ray in the tree $\T \join \T^+_{\neq i}$ distributed according to the $\lambda$-harmonic measure on the tree boundary. Then $\mathcal{R}^+$ can be identified with $\mathcal{R}_{\neq i}$ conditionally on $\{\mathcal{R}_{\neq i}\in \partial \T\}$. We see that $I$ is equal to
\begin{align*}
& E\bigg[\frac{\beta(\T)}{\lambda-1+\beta(\T)+\cc(\T^+)} \sum_{i=1}^{\nu^+} F(\T^+_i)G(\T \join \T^+_{\neq i}, \mathcal{R}_{\neq i})\ind{\{\mathcal{R}_{\neq i}\in \partial \T\}}\frac{\cc(\T \join \T^+_{\neq i})}{\beta(\T)}\bigg]\\
=\,\,& E\bigg[\frac{1}{\lambda-1+\beta(\T)+\cc(\T^+)} \sum_{i=1}^{\nu^+}  F(\T^+_i)G(\T \join \T^+_{\neq i}, \mathcal{R}_{\neq i})\ind{\{\mathcal{R}_{\neq i}\in \partial \T\}} \cc(\T \join \T^+_{\neq i})\bigg].
\end{align*}
By symmetry, we deduce further that 
\begin{align*}
I  &= E\bigg[ \frac{\cc(\T \join \T^+_{\neq 1})}{\lambda-1+\beta(\T)+\cc(\T^+)} F(\T^+_1)G(\T \join \T^+_{\neq 1},\mathcal{R}_{\neq 1})\Big(\ind{\{\mathcal{R}_{\neq 1}\in \partial \T\}} + \sum_{i=2}^{\nu^+} \ind{\{\mathcal{R}_{\neq 1}\in \partial \T^+_i\}}\Big) \bigg] \\
&= E\bigg[ \frac{\cc(\T \join \T^+_{\neq 1})}{\lambda-1+\beta(\T)+\cc(\T^+)} F(\T^+_1)G(\T \join \T^+_{\neq 1},\mathcal{R}_{\neq 1})\bigg].
\end{align*}
As $\beta(\T)+\cc(\T^+)=\beta(\T)+\sum_{i=1}^{\nu^+}\beta(\T^+ _i)=\beta(\T^+_1)+\cc(\T \join \T^+_{\neq 1})$, we obtain from the previous display that 
\begin{equation*}
I = E\bigg[ \frac{\cc(\T)}{\lambda-1+\beta(\T^+)+\cc(\T)} F(\T^+)G(\T, \mathcal{R})\bigg].
\end{equation*}
Using \eqref{eq:cond-symmetry} and \eqref{eq:conductance}, we get therefore
\begin{align*}
I &= E\bigg[ \frac{\beta(\T)}{\lambda-1+\beta(\T)+\cc(\T^+)}\frac{\cc(\T^+)}{\beta(\T^+)} F(\T^+)G(\T,\mathcal{R})\bigg]\\
&= E\bigg[ \frac{\beta(\T) (\lambda +\cc(\T^+))}{\lambda-1+\beta(\T)+\cc(\T^+)} F(\T^+)G(\T,\mathcal{R})\bigg].
\end{align*}

{\bf Case II}: The vertex $e$ becomes the new root of $\tilde \T \join \tilde \T^+$, which happens with probability $\lambda/(\nu^+ +\lambda)$. 
In this case, if $\mathcal{R}$ passes through the root of $\T_k$ for some integer $k\in [1,\nu]$, then
\begin{equation*}
\tilde \T= \T_k \quad \mbox{and} \quad \tilde \T^+= \T^+ \join \T_{\neq k},
\end{equation*}
where $\T_{\neq k}$ stands for the tree rooted at $e$ formed by all descendant trees $\{\T_\ell, 1\leq \ell\leq \nu, \ell\neq k\}$ together with the edges connecting their roots to $e$. 
As in the previous case, $\T_k$ and $\T^+ \join \T_{\neq k}$ are two independent Galton--Watson trees. But $\tilde{\mathcal{R}}$ is now the ray $\mathcal{R}^-$ obtained by deleting $e$ from the beginning of the sequence $\mathcal{R}$. We set thus
\begin{align*}
I\!I &  \colonequals E\bigg[ \frac{(\lambda+\nu^+) \beta(\T)}{\lambda-1+\beta(\T)+\cc(\T^+)} \frac{\lambda}{\nu^+ +\lambda} \sum_{k=1}^{\nu} F(\T^+ \join \T_{\neq k})G(\T_k, \mathcal{R}^-)\ind{\{\mathcal{R}^-\in \partial \T_k\}}\bigg] \\
&= E\bigg[ \frac{\lambda \beta(\T)}{\lambda-1+\beta(\T)+\cc(\T^+)} \sum_{k=1}^{\nu} F(\T^+ \join \T_{\neq k})G(\T_k, \mathcal{R}^-)\ind{\{\mathcal{R}^-\in \partial \T_k\}}\bigg].
\end{align*}
Given $\T$ and $\T^+$, we let $\mathcal{R}_{k}$ be a random ray in the tree $\T_k$ distributed according to the $\lambda$-harmonic measure. It follows that
\begin{align*}
I\!I & = E\bigg[ \sum_{k=1}^{\nu} \frac{\lambda \beta(\T)}{\lambda-1+\beta(\T)+\cc(\T^+)} F(\T^+ \join \T_{\neq k})G(\T_k, \mathcal{R}_k)\frac{\beta(\T_k)}{\cc(\T)}\bigg] \\
&= E\bigg[ \sum_{k=1}^{\nu} \frac{\beta(\T_k)}{(\lambda-1+\beta(\T)+\cc(\T^+))(1+\lambda^{-1} \cc(\T))} F(\T^+ \join \T_{\neq k})G(\T_k, \mathcal{R}_k)\bigg].
\end{align*}
Using the identity \eqref{eq:cond-identity}, we see that  
\begin{eqnarray*}
(\lambda-1+\beta(\T)+\cc(\T^+))(1+\lambda^{-1} \cc(\T))&= & \big(\lambda-1+\beta(\T^+)+\cc(\T)\big)\big(1+\lambda^{-1} \cc(\T^+)\big)\\
&=& \big(\lambda-1+\beta(\T_k)+\cc(\T^+ \join \T_{\neq k})\big)\big(1+\lambda^{-1} \cc(\T^+)\big).
\end{eqnarray*}
Together with \eqref{eq:conductance}, it implies 
\begin{align*}
I\!I & = E\bigg[ \sum_{k=1}^{\nu} \frac{\beta(\T_k)(1+\lambda^{-1} \cc(\T^+))^{-1}}{\lambda-1+\beta(\T_k)+\cc(\T^+ \join \T_{\neq k})} F(\T^+ \join \T_{\neq k})G(\T_k, \mathcal{R}_k)\bigg] \\
& = E\bigg[ \sum_{k=1}^{\nu} \frac{\beta(\T_k)(1-\beta(\T^+))}{\lambda-1+\beta(\T_k)+\cc(\T^+ \join \T_{\neq k})} F(\T^+ \join \T_{\neq k})G(\T_k, \mathcal{R}_k)\bigg] .
\end{align*}
Observe that the root of $\T^+ \join \T_{\neq k}$ has $\nu$ children. 
For any integer $m\geq k$, the conditional law of $(\T_k, \T^+ \join \T_{\neq k})$ given $\{\nu=m\}$ is the same as that of $(\T,\T^+)$ conditionally on $\{\nu^+=m\}$.
Hence, we obtain
\begin{align*}
I\!I & = E\bigg[ \sum_{k=1}^{\nu^+} \frac{\beta(\T)(1-\beta(\T_k^+))}{\lambda-1+\beta(\T)+\cc(\T^+)} F(\T^+)G(\T, \mathcal{R})\bigg]  \\
& = E\bigg[\frac{\beta(\T)(\nu^+ -\cc(\T^+))}{\lambda-1+\beta(\T)+\cc(\T^+)} F(\T^+)G(\T, \mathcal{R})\bigg] .
\end{align*}

Finally, adding up Cases I and II, we have 
\begin{align*}
& E\Big[\frac{(\lambda+\nu^+) \beta(\T)}{\lambda-1+\beta(\T)+\cc(\T^+)} F(\tilde \T^+)G(\tilde \T, \tilde{\mathcal{R}})\Big] \\
&= E\bigg[ \frac{\beta(\T) (\lambda +\cc(\T^+))}{\lambda-1+\beta(\T)+\cc(\T^+)} F(\T^+)G(\T,\mathcal{R})\bigg]+ E\bigg[\frac{\beta(\T)(\nu^+ -\cc(\T^+))}{\lambda-1+\beta(\T)+\cc(\T^+)} F(\T^+)G(\T, \mathcal{R})\bigg]\\
&= E\Big[\frac{(\lambda+\nu^+) \beta(\T)}{\lambda-1+\beta(\T)+\cc(\T^+)} F(\T^+)G(\T, \mathcal{R})\Big],
\end{align*}
which completes the proof of the stationarity.
\end{proof}

We write $\overset \rightarrow x$ for an infinite path $(x_n)_{n\geq 0}$ in $T$. 
Let $\mathsf{RW}_\lambda\times \mathbf{AGW}_\lambda$ be the probability measure on the space
\begin{equation*}
\big\{(\overset \rightarrow x,(T,\xi))\mid (T,\xi)\in \t^*, \overset \rightarrow x\subset T \big\}
\end{equation*} 
that is associated to the Markov chain considered in Proposition~\ref{prop:invariance}. It is given by choosing a tree $T$ with a marked ray $\xi$ according to $\mathbf{AGW}_\lambda$, and then independently running on $T$ a random walk $\lambda$-biased towards $\xi$.

\section{Harmonic-stationary measure}
\label{sec:harm-invar}

Let $\mathrm{HARM}^T_\lambda$ be the flow on the vertices of $T$ in correspondence with the $\lambda$-harmonic measure on $\partial T$, so that $\mathrm{HARM}_\lambda^T(u)$ coincides with the mass given by the $\lambda$-harmonic measure to the set of all rays passing through the vertex $u$.
We denote by $\mathsf{HARM}_\lambda$ the transition probabilities for a Markov chain on $\t$, that goes from a tree $T$ to the descendant tree $T_u$, $|u|=1$, with probability 
\begin{equation*}
\mathrm{HARM}_\lambda^T(u)= \frac{\beta(T_u)}{\sum_{|w|=1}\beta(T_w)} = \frac{\beta(T_u)}{\cc(T)}.
\end{equation*}

The existence of a $\mathsf{HARM}_\lambda$-stationary probability measure $\mu_{\mathsf{HARM}_\lambda}$ that is absolutely continuous with respect to $\mathbf{GW}$ was established in Lemma 5.2 of \cite{LPP96}. Taking into account the stationary measure of the environment $\mathbf{AGW}_\lambda$, we can construct $\mu_{\mathsf{HARM}_\lambda}$ as an induced measure by considering the $\lambda$-biased random walk at the exit epochs. See~\cite[Section 8]{LPP95} and~\cite[Section 5]{LPP96} for more details.

According to Proposition 5.2 of \cite{LPP95}, $\mu_{\mathsf{HARM}_\lambda}$ is equivalent to $\mathbf{GW}$ and the associated $\mathsf{HARM}_\lambda$-Markov chain is ergodic. Ergodicity implies further that $\mu_{\mathsf{HARM}_\lambda}$ is the unique $\mathsf{HARM}_\lambda$-stationary probability measure absolutely continuous with respect to $\mathbf{GW}$.
Due to uniqueness, we can identify $\mu_{\mathsf{HARM}_\lambda}$ via the next result. 

\begin{lemma}
\label{lemma:harm-invariance}
For every $x>0$, set
\begin{equation*}
\kappa_\lambda(x)\colonequals \mathbf{GW}\bigg[\frac{\beta(T)x}{\lambda-1+\beta(T)+x}\bigg]=
E\bigg[\frac{\beta(\T)x}{\lambda-1+\beta(\T)+x}\bigg].
\end{equation*}
The finite measure $\kappa_\lambda(\cc(T))\mathbf{GW}(\mathrm{d}T)$ is $\mathsf{HARM}_\lambda$-stationary.
\end{lemma}

\begin{proof}
The function $\kappa_\lambda \colon \R^+ \to \R^+$ is bounded and strictly increasing. 
In fact, for $\mathbf{GW}$-a.e.~$T$, $\lambda-1+\beta(T)> 0$. The function
\begin{equation*}
\frac{\beta(T)x}{\lambda-1+\beta(T)+x}
\end{equation*}
is strictly increasing in $x$, and it is bounded above by $\beta(T)$.
Thus, $\kappa_\lambda(x) < \mathbf{GW}[\beta(T)]< 1$.

We write $\nu$ for the offspring number of the root of $T$. Conditionally on the event $\{\nu=k\}$, let $T_1,\ldots, T_k$ denote the descendant trees of the children of the root. In order to prove the $\mathsf{HARM}_\lambda$-stationarity, we must verify that for any bounded measurable function $F$ on $\t$, the integral $\int F(T)\kappa_\lambda(\cc(T))\mathbf{GW}(\mathrm{d}T)$ is equal to 
\begin{align*}
I \colonequals & \sum_{k=1}^\infty p_k \sum_{i=1}^k \int F(T_i)\kappa_\lambda(\cc(T))\frac{\beta(T_i)}{\beta(T_1)+\cdots+\beta(T_k)} \mathbf{GW}(\mathrm{d}T\mid \nu=k) \\
= & \sum_{k=1}^\infty k p_k \int F(T_1)\kappa_\lambda(\cc(T))\frac{\beta(T_1)}{\beta(T_1)+\cdots+\beta(T_k)} \mathbf{GW}(\mathrm{d}T\mid \nu=k).
\end{align*}
Using the definition of $\kappa_\lambda$ and the branching property, we see that $I$ is given by 
\begin{align*}
& \sum_{k=1}^\infty k p_k \int F(T_1)\frac{\beta(T_0)\beta(T_1)}{\lambda-1+\beta(T_0)+\beta(T_1)+\cdots+\beta(T_k)} \mathbf{GW}(\mathrm{d}T\mid \nu=k)\mathbf{GW}(\mathrm{d}T_0) \\
=& \sum_{k=1}^\infty k p_k \int F(T_1)\frac{\beta(T_0)\beta(T_1)}{\lambda-1+\beta(T_0)+\beta(T_1)+\cdots+\beta(T_k)} \mathbf{GW}(\mathrm{d}T_0)\mathbf{GW}(\mathrm{d}T_1)\cdots \mathbf{GW}(\mathrm{d}T_k) \\
=& \sum_{k=1}^\infty p_k \int F(T_1)\frac{\beta(T_1)(\beta(T_0)+\beta(T_2)+\cdots+\beta(T_k))}{\lambda-1+\beta(T_0)+\beta(T_1)+\cdots+\beta(T_k)} \mathbf{GW}(\mathrm{d}T_0)\mathbf{GW}(\mathrm{d}T_1)\cdots \mathbf{GW}(\mathrm{d}T_k) \\
=& \int F(T_1) \frac{\beta(T_1)\cc(T)}{\lambda-1+\beta(T_1)+\cc(T)}\mathbf{GW}(\mathrm{d}T)\mathbf{GW}(\mathrm{d}T_1).
\end{align*}
Hence, it follows from \eqref{eq:cond-symmetry} that 
\begin{equation*}
I= \int F(T_1) \frac{\beta(T)\cc(T_1)}{\lambda-1+\beta(T)+\cc(T_1)}\mathbf{GW}(\mathrm{d}T) \mathbf{GW}(\mathrm{d}T_1) =\int F(T_1)\kappa_\lambda(\cc(T_1))\mathbf{GW}(\mathrm{d}T_1),
\end{equation*}
which finishes the proof. 
\end{proof}

We deduce from the preceding lemma that the Radon--Nikod\'ym derivative of $\mu_{\mathsf{HARM}_\lambda}$ with respect to $\mathbf{GW}$ is a.s. 
\begin{equation}
\label{eq:density-Harm}
\frac{\mathrm{d}\mu_{\mathsf{HARM}_\lambda}}{\mathrm{d}\mathbf{GW}} (T)= \frac{1}{h_\lambda} \kappa_\lambda(\cc(T))=
\frac{1}{h_\lambda} \int \frac{\beta(T')\cc(T)}{\lambda-1+\beta(T')+\cc(T)} \mathbf{GW}(\mathrm{d}T'),
\end{equation}
where the normalizing constant 
\begin{equation*}
h_\lambda = \int \frac{\beta(T')\cc(T)}{\lambda-1+\beta(T')+\cc(T)} \mathbf{GW}(\mathrm{d}T)\mathbf{GW}(\mathrm{d}T') = E\bigg[\frac{\beta(\T)\cc(\T^+)}{\lambda-1+\beta(\T)+\cc(\T^+)} \bigg].
\end{equation*}
Writing $\mathcal{R}(T)=\cc(T)^{-1}$ for the effective resistance, one can reformulate \eqref{eq:density-Harm} as 
\begin{equation*}
\frac{\mathrm{d}\mu_{\mathsf{HARM}_\lambda}}{\mathrm{d}\mathbf{GW}} (T)= 
\frac{1}{h_\lambda} \int \frac{\lambda^{-1}}{(\lambda-1)\mathcal{R}(T)\mathcal{R}(T')+\mathcal{R}(T)+\mathcal{R}(T')+\lambda^{-1}} \mathbf{GW}(\mathrm{d}T').
\end{equation*}
When $\lambda=1$, it coincides with the expression of the same density in Section 8 of~\cite{LPP95}.

As we can see in the proof of Lemma~\ref{lemma:harm-invariance}, the mesure $\mu_{\mathsf{HARM}_\lambda}$ defined by~\eqref{eq:density-Harm} is still $\mathsf{HARM}_\lambda$-stationary when $p_0>0$ is allowed. 
We also point out that the proof of Proposition 17.31 in \cite{LP-book} (corresponding to the case $\lambda=1$) can be adapted to derive \eqref{eq:density-Harm} from the construction of $\mu_{\mathsf{HARM}_\lambda}$ by inducing. 

In a recent work \cite{Rou}, Rousselin develops a general result to construct explicit stationary measures for a certain class of Markov chains on trees. Applying his result to the $\mathsf{HARM}_\lambda$-Markov chain considered above gives the same formula \eqref{eq:density-Harm}, see Theorem 4.1 in \cite{Rou}.

\section{Dimension of the harmonic measure}
\label{sec:dim-harm}
Let $\mathsf{T}$ be a random tree distributed as $\mu_{\mathsf{HARM}_\lambda}$, and let $\Theta$ be the $\lambda$-harmonic ray in $\mathsf{T}$. If we denote the vertices along $\Theta$ by $\Theta_0, \Theta_1, \ldots$, then according to the flow property of harmonic measure, the sequence of descendant trees $(\mathsf{T}_{\Theta_n})_{n\geq 0}$ is a stationary $\mathsf{HARM}_\lambda$-Markov chain. 
In what follows, we write $\mathsf{HARM}_\lambda\times \mu_{\mathsf{HARM}_\lambda}$ for the law of $(\Theta, \mathsf{T})$ on the space $\{(\xi, T)\mid T\in \t, \xi \in \partial T\}$. 
Recall that the ergodicity of $\mathsf{HARM}_\lambda\times \mu_{\mathsf{HARM}_\lambda}$ results from Proposition 5.2 in \cite{LPP95}.

As shown in \cite[Section 5]{LPP95}, the Hausdorff dimension $d_\lambda$ of the $\lambda$-harmonic measure coincides with the entropy 
\begin{equation*}
\mathrm{Entropy}_{\mathsf{HARM}_\lambda}(\mu_{\mathsf{HARM}_\lambda}) \colonequals 
\int \log \frac{1}{\mathrm{HARM}_\lambda^T(\xi_1)} \mathsf{HARM}_\lambda\times \mu_{\mathsf{HARM}_\lambda}(\mathrm{d}\xi,\mathrm{d}T).
\end{equation*}
Thus, by \eqref{eq:conductance} we have 
\begin{eqnarray*}
d_\lambda &= & \int \log \frac{\cc(T)}{\beta(T_{\xi_1})} \mathsf{HARM}_\lambda\times \mu_{\mathsf{HARM}_\lambda}(\mathrm{d}\xi,\mathrm{d}T)\\
&=& \int \log \frac{\lambda\beta(T)}{\beta(T_{\xi_1})(1-\beta(T))} \mathsf{HARM}_\lambda\times \mu_{\mathsf{HARM}_\lambda}(\mathrm{d}\xi,\mathrm{d}T).
\end{eqnarray*}
By stationarity, 
\begin{equation}
\label{eq:dim-formula}
d_\lambda = \int \log \frac{\lambda}{1-\beta(T)} \mu_{\mathsf{HARM}_\lambda}(\mathrm{d}T) = \int \log \big(\cc(T)+\lambda \big) \mu_{\mathsf{HARM}_\lambda}(\mathrm{d}T),
\end{equation}
provided the integral $\int \log \beta(T)^{-1} \mu_{\mathsf{HARM}_\lambda}(\mathrm{d}T)$ is finite. 
Using the explicit form \eqref{eq:density-Harm} of $\mu_{\mathsf{HARM}_\lambda}$, we see that this integral is equal to 
\begin{equation*}
h_\lambda^{-1} E \bigg[ \frac{\beta(\T)\cc(\T^+)}{\lambda-1+\beta(\T)+\cc(\T^+)} \log \frac{1}{\beta(\T^+)}\bigg],
\end{equation*}
in which the expectation is less than 
\begin{equation}
\label{eq:product-expectation}
E \bigg[ \frac{\beta(\T)}{\lambda-1+\beta(\T)} \Big(\cc(\T^+)\log \frac{1}{\beta(\T^+)}\Big)\bigg]= E\bigg[\frac{\beta(\T)}{\lambda-1+\beta(\T)} \bigg] \cdot
E \bigg[ \frac{\lambda\beta(\T^+)}{1-\beta(\T^+)} \log \frac{1}{\beta(\T^+)}\bigg].
\end{equation}
Notice that for $x\in (0,1)$, 
\begin{equation*}
0<\frac{x}{1-x}\log\frac{1}{x}<1.
\end{equation*}
Hence, the product in \eqref{eq:product-expectation} is bounded by 
\begin{equation*}
\mathbf{GW}\bigg[\frac{\lambda \beta(T)}{\lambda-1+\beta(T)} \bigg],
\end{equation*}
which is finite according to \eqref{eq:Aid-lemma}. 
Therefore, the formula \eqref{eq:dim-formula} is justified. By \eqref{eq:density-Harm} again, we obtain
\begin{eqnarray*}
d_\lambda &=& h_\lambda^{-1} \int \log \big(\cc(T)+\lambda\big) \frac{\beta(T')\cc(T)}{\lambda-1+\beta(T')+\cc(T)} \mathbf{GW}(\mathrm{d}T)\mathbf{GW}(\mathrm{d}T') \\
&=& h_\lambda^{-1} E\bigg[\log (\cc(\T^+)+\lambda) \frac{\beta(\T)\cc(\T^+)}{\lambda-1+\beta(\T)+\cc(\T^+)} \bigg].
\end{eqnarray*}

Now let us prove Theorem~\ref{thm:dim-harm} by first showing $d_\lambda>\mathbf{GW}[\log \nu]$. Recall that the function
$\kappa_\lambda$ is strictly increasing.
The FKG inequality implies that 
\begin{equation*}
E\bigg[\log (\cc(\T^+)+\lambda) \frac{\beta(\T)\cc(\T^+)}{\lambda-1+\beta(\T)+\cc(\T^+)} \bigg] > E\big[\log (\cc(\T^+)+\lambda)\big] \times E\bigg[\frac{\beta(\T)\cc(\T^+)}{\lambda-1+\beta(\T)+\cc(\T^+)} \bigg].
\end{equation*}
In view of the previous formula for $d_\lambda$, it suffices to prove 
\begin{equation*}
\mathbf{GW}[\log (\cc(T)+\lambda)]\geq \mathbf{GW}[\log \nu].
\end{equation*}
In fact, the strict inequality holds. Recall the notation that $T_1,\ldots, T_\nu$ stand for the descendant trees of the children of the root in $T$, and notice that 
\begin{equation*}
\cc(T)+\lambda=\frac{\cc(T)}{\beta(T)}=\frac{\sum_{i=1}^\nu \beta(T_i)}{\beta(T)}.
\end{equation*}
By strict concavity of the log function, 
\begin{equation*}
\log \sum_{i=1}^{\nu} \beta(T_i) \geq \frac{1}{\nu} \sum_{i=1}^{\nu} \log (\nu \beta(T_i))=\log \nu + \frac{1}{\nu} \sum_{i=1}^{\nu} \log \beta(T_i)
\end{equation*}
with equality if and only if all $\beta(T_i), 1\leq i\leq \nu,$ are equal. 
But this condition for equality cannot hold for $\mathbf{GW}$-almost every $T$. 
Meanwhile, it follows from Lemma~\ref{lem:mean-resistance-log} that 
\begin{equation*}
\mathbf{GW}\left[\frac{1}{\nu} \sum_{i=1}^{\nu} \log \beta(T_i)\right]=\mathbf{GW}[\log \beta(T)].
\end{equation*}
Therefore, 
\begin{equation*}
\mathbf{GW}[\log (\cc(T)+\lambda)] > \mathbf{GW}[\log \nu].
\end{equation*}

To complete the proof of Theorem~\ref{thm:dim-harm}, it remains to examine the asymptotic behaviors of $d_\lambda$.
When $\lambda\to 0^+$, a.s.~$\beta(\T)\to 1$ and $\cc(\T^+)=\sum_{i=1}^{\nu^+} \beta(\T^+_i) \to \nu^+$. 
Since 
\begin{equation}
\label{eq:upper-bd-harm-density}
\frac{\beta(\T)\cc(\T^+)}{\lambda-1+\beta(\T)+\cc(\T^+)} \leq \beta(\T) \leq 1,
\end{equation}
we can use Lebesgue's dominated convergence to get $\lim_{\lambda \to 0^+} h_\lambda= 1$.
Similarly, it follows from 
\begin{equation*}
\log(\cc(\T^+)+\lambda)\frac{\beta(\T)\cc(\T^+)}{\lambda-1+\beta(\T)+\cc(\T^+)} \leq \log(\cc(\T^+)+\lambda) \leq \log(\nu^+ +m)
\end{equation*}
that $\lim_{\lambda \to 0^+} d_\lambda = E[\log \nu^+]= \mathbf{GW}[\log \nu]$.

When $\lambda\to m^-$, a.s.~$\beta(\T)\to 0$ and $\cc(\T^+)\to 0$. We have seen that the FKG inequality yields the lower bound
\begin{equation*}
d_\lambda > E[\log(\cc(\T^+)+\lambda)].
\end{equation*}
Using again dominated convergence, we obtain
\begin{equation*}
\lim_{\lambda \to m^-}  E[\log(\cc(\T^+)+\lambda)]=\log m.
\end{equation*}
On the other hand, recall that $d_\lambda <\log m$.
Consequently, $d_\lambda \to \log m$ when $\lambda \to m^-$.

\section{Average number of children along a random path}
\label{sec:average-nb-child}
Recall that for every vertex $x$ in a tree $T$, we write $\nu(x)$ for its number of children. 
Birkhoff's ergodic theorem implies that for $\mathsf{HARM}_\lambda\times \mu_{\mathsf{HARM}_\lambda}$-a.e.~$(\xi, T)$, 
\begin{equation*}
\lim_{n\to \infty} \frac{1}{n} \sum_{k=0}^{n-1} \nu(\xi_k)= \int \nu(e)\mu_{\mathsf{HARM}_\lambda}(\mathrm{d}T)= h_\lambda^{-1} E\bigg[ \frac{\nu^+ \beta(\T)\cc(\T^+)}{\lambda-1+\beta(\T)+\cc(\T^+)} \bigg].
\end{equation*}
The last expectation is finite, as we derive from \eqref{eq:upper-bd-harm-density} that
\begin{equation*}
\frac{\nu^+ \beta(\T)\cc(\T^+)}{\lambda-1+\beta(\T)+\cc(\T^+)} \leq \nu^+ .
\end{equation*}
Since $\mu_{\mathsf{HARM}_\lambda}$ is equivalent to $\mathbf{GW}$, the convergence above also holds for $\mathsf{HARM}_\lambda\times \mathbf{GW}$-a.e.~$(\xi, T)$. Hence, the average number of children of the vertices visited by the $\lambda$-harmonic ray in a Galton--Watson tree is the same as the $\mu_{\mathsf{HARM}_\lambda}$-mean degree of the root. 

For every $k\geq 1$, we set 
\begin{equation*}
A(k) \colonequals E\bigg[ \frac{\beta(\T)\sum_{i=1}^{k}\beta(\T_i^+)}{\lambda-1+\beta(\T)+\sum_{i=1}^{k}\beta(\T_i^+)} \bigg].
\end{equation*}
The sequence $(A(k))_{k\geq 1}$ is strictly increasing. Moreover, 
\begin{equation*}
\frac{A(k)}{k}=E\bigg[ \frac{\beta(\T)\beta(\T_1^+)}{\lambda-1+\beta(\T)+\sum_{i=1}^{k}\beta(\T_i^+)} \bigg]
\end{equation*}
is strictly decreasing with respect to $k$.

\begin{proposition}
\label{prop:child-harm-visi}
For $0<\lambda<m$, 
\begin{equation*}
\int \nu(e)\mu_{\mathsf{HARM}_\lambda}(\mathrm{d}T) > m. 
\end{equation*}
Furthermore, $\int \nu(e)\mu_{\mathsf{HARM}_\lambda}(\mathrm{d}T) \to m$ as $\lambda\to 0^+$.
\end{proposition}

\begin{proof}
The first assertion, reformulated as 
\begin{equation*}
E\bigg[ \frac{\nu^+ \beta(\T)\sum_{i=1}^{\nu^+}\beta(\T_i^+)}{\lambda-1+\beta(\T)+\sum_{i=1}^{\nu^+}\beta(\T_i^+)} \bigg] > E[\nu^+] \cdot E\bigg[ \frac{\beta(\T)\sum_{i=1}^{\nu^+}\beta(\T_i^+)}{\lambda-1+\beta(\T)+\sum_{i=1}^{\nu^+}\beta(\T_i^+)} \bigg],
\end{equation*}
is a simple consequence of the FKG inequality, since
\begin{equation*}
\mathbf{GW}[\nu A(\nu)]>\mathbf{GW}[\nu] \cdot \mathbf{GW}[A(\nu)].
\end{equation*}

When $\lambda\to 0^+$, a.s.~$\beta(\T)\to 1$ and $\cc(\T^+)\to \nu^+$. Using Lebesgue's dominated convergence, we have seen at the end of Section~\ref{sec:dim-harm} that $\lim_{\lambda \to 0^+} h_\lambda= 1$.
The same argument applies to the convergence of 
\begin{equation*}
E\bigg[ \frac{\nu^+ \beta(\T)\cc(\T^+)}{\lambda-1+\beta(\T)+\cc(\T^+)} \bigg]
\end{equation*}
towards $E[\nu^+]=m$. 
\end{proof}

Under $\mathbf{GW}$ we define a random variable $\hat \nu$ having the size-biased distribution of $\nu$.

\begin{proposition}
\label{prop:child-harm-unif}
For $0<\lambda<m$, 
\begin{equation*}
\int \nu(e)\mu_{\mathsf{HARM}_\lambda}(\mathrm{d}T) < \mathbf{GW}[\hat\nu]= m^{-1}\textstyle \sum k^2 p_k. 
\end{equation*}
If we assume further that $\sum k^3 p_k<\infty$, then $\int \nu(e)\mu_{\mathsf{HARM}_\lambda}(\mathrm{d}T) \to \mathbf{GW}[\hat\nu]$ as $\lambda\to m^-$.
\end{proposition}

\begin{proof}
Since $\int \nu(e)\mu_{\mathsf{HARM}_\lambda}(\mathrm{d}T) < \infty$, we may assume $\sum k^2 p_k<\infty$ throughout the proof.
The inequality in the first assertion can be written as 
\begin{equation*}
E[\nu^+] \cdot E\bigg[\frac{\nu^+ \beta(\T)\sum_{i=1}^{\nu^+}\beta(\T_i^+)}{\lambda-1+\beta(\T)+\sum_{i=1}^{\nu^+}\beta(\T_i^+)} \bigg]< E[(\nu^+)^2] \cdot E\bigg[\frac{\beta(\T)\sum_{i=1}^{\nu^+}\beta(\T_i^+)}{\lambda-1+\beta(\T)+\sum_{i=1}^{\nu^+}\beta(\T_i^+)} \bigg].
\end{equation*}
By conditioning on $\nu^+$, we see that it is equivalent to 
\begin{equation*}
\mathbf{GW}[A(\hat \nu)]< \mathbf{GW}[\hat \nu] \cdot \mathbf{GW}\bigg[\frac{A(\hat \nu)}{\hat \nu}\bigg],
\end{equation*}
which results from the FKG inequality. 

For the second assertion, remark that 
\begin{eqnarray*}
E\bigg[ \frac{\nu^+ \beta(\T)\sum_{i=1}^{\nu^+}\beta(\T_i^+)}{m-1} \bigg] &=& 
\frac{\mathbf{GW}[\nu^2]\cdot \mathbf{GW}[\beta(T)]^2}{m-1} ,\\
E\bigg[ \frac{\beta(\T)\sum_{i=1}^{\nu^+}\beta(\T_i^+)}{m-1} \bigg] &=& 
\frac{\mathbf{GW}[\nu]\cdot \mathbf{GW}[\beta(T)]^2}{m-1}.
\end{eqnarray*}
When the offspring distribution $p$ admits a second moment, Proposition 3.1 of \cite{BHOZ} shows that 
\begin{equation*}
\frac{\beta(T)}{\mathbf{GW}[\beta(T)]}
\end{equation*}
is uniformly bounded in $L^2(\mathbf{GW})$. Using this fact, we can verify that
\begin{equation*}
\lim_{\lambda\to m^-}  h_\lambda \cdot E\bigg[ \frac{\beta(\T)\sum_{i=1}^{\nu^+}\beta(\T_i^+)}{m-1} \bigg]^{-1} = 1.
\end{equation*}
With the third moment condition $\sum k^3 p_k<\infty$, we similarly have
\begin{equation*}
\lim_{\lambda\to m^-} E\bigg[ \frac{\nu^+ \beta(\T)\sum_{i=1}^{\nu^+}\beta(\T_i^+)}{\lambda-1+\beta(\T)+\sum_{i=1}^{\nu^+}\beta(\T_i^+)} \bigg] \cdot E\bigg[ \frac{\nu^+ \beta(\T)\sum_{i=1}^{\nu^+}\beta(\T_i^+)}{m-1} \bigg]^{-1}=1.
\end{equation*}
Therefore,
\begin{equation*}
\int \nu(e)\mu_{\mathsf{HARM}_\lambda}(\mathrm{d}T) \to \frac{\mathbf{GW}[\nu^2]}{\mathbf{GW}[\nu]}
=\mathbf{GW}[\hat\nu]
\end{equation*}
as $\lambda\to m^-$.
\end{proof}

Now we turn to investigate the average number of children seen by the $\lambda$-biased random walk. First of all, as remarked in \cite[Section 8]{LPP95}, the ergodicity of $\mathsf{HARM}_\lambda\times \mu_{\mathsf{HARM}}$ implies that $\mathsf{RW}_\lambda\times \mathbf{AGW}_\lambda$ is also ergodic.
For a tree $T$ rooted at $e$, let $\nu^+(e)$ denote the number of children of the root minus 1. 
Since 
\begin{equation*}
E\bigg[\frac{\nu^+(\lambda+\nu^+) \beta(\T)}{\lambda-1+\beta(\T)+\cc(\T^+)} \bigg]
= E\bigg[\frac{(\lambda+\nu^+)\cc(\T^+)}{\lambda-1+\beta(\T)+\cc(\T^+)} \bigg]  \leq \lambda+ E[\nu^+]<\infty,
\end{equation*}
it follows from Birkhoff's ergodic theorem that for $\mathsf{RW}_\lambda\times \mathbf{AGW}_\lambda$-a.e.~$(\overset\rightarrow x, (T,\xi))$, 
\begin{equation}
\label{eq:average-nb-child-walk}
\lim_{n\to \infty} \frac{1}{n} \sum_{k=0}^{n-1} \nu(x_k)= \int \nu^+(e) \mathbf{AGW}_\lambda(\mathrm{d}T,\mathrm{d}\xi)= c_\lambda^{-1} E\bigg[\frac{\nu^+(\lambda+\nu^+) \beta(\T)}{\lambda-1+\beta(\T)+\cc(\T^+)} \bigg].
\end{equation}
Using arguments similar to those in the last remark on page 600 of \cite{LPP95}, we deduce that the average number of children seen by the $\lambda$-biased random walk on $\T$ is a.s.~given by the same integral $\int \nu^+(e) \mathbf{AGW}_\lambda(\mathrm{d}T,\mathrm{d}\xi)$.

\begin{proposition}
\label{prop:child-rw-path}
We have  
\begin{equation*}
\int \nu^+(e) \mathbf{AGW}_\lambda(\mathrm{d}T,\mathrm{d}\xi) \quad \left\{
   \begin{aligned}
   <m &\quad \mbox{when $0< \lambda < 1$};  \\
   =m &\quad \mbox{when $\lambda \in \{0,1\}$}; \\
   >m &\quad \mbox{when $1< \lambda <m$}. \\
   \end{aligned}
  \right.
\end{equation*}

\end{proposition}

\begin{proof}
For every integer $k\geq 1$ we set 
\begin{equation*}
B_\lambda(k) \colonequals E\bigg[ \frac{(\lambda+k)\beta(\T)}{\lambda-1+\beta(\T)+\sum_{i=1}^{k}\beta(\T_i^+)} \bigg].
\end{equation*}
Clearly, we have
\begin{equation*}
\int \nu^+(e) \mathbf{AGW}_\lambda(\mathrm{d}T,\mathrm{d}\xi)= m \frac{\mathbf{GW}[B_\lambda(\hat \nu)]}{\mathbf{GW}[B_\lambda(\nu)]}.
\end{equation*} 
When $\lambda\in \{0,1\}$, $B_\lambda(k)= 1$ for all $k$. We will show that the sequence $(B_\lambda(k))_{k\geq 1}$ is strictly decreasing when $0<\lambda <1$, and strictly increasing when $1<\lambda<m$. Therefore, by the FKG inequality, $\mathbf{GW}[B_\lambda(\hat \nu)]>\mathbf{GW}[B_\lambda(\nu)]$ when $1<\lambda<m$, and $\mathbf{GW}[B_\lambda(\hat \nu)]<\mathbf{GW}[B_\lambda(\nu)]$ when $0<\lambda <1$. 

To get the claimed monotonicity of the sequence $(B_\lambda(k))_{k\geq 1}$, notice that 
\begin{equation*}
B_\lambda(k+1)= E\bigg[ \frac{(\lambda+k)\beta(\T)+\beta(\T^+_{k+1})}{\lambda-1+\beta(\T)+\sum_{i=1}^{k+1}\beta(\T_i^+)} \bigg].
\end{equation*}
Simple calculations give
\begin{align*}
B_\lambda(k+1)-B_\lambda(k) & = E\bigg[ \frac{-(\lambda+k)\beta(\T)\beta(\T^+_{k+1})+\beta(\T^+_{k+1})(\lambda-1+\beta(\T)+\sum_{i=1}^{k}\beta(\T_i^+))}{(\lambda-1+\beta(\T)+\sum_{i=1}^{k+1}\beta(\T_i^+))(\lambda-1+\beta(\T)+\sum_{i=1}^{k}\beta(\T_i^+))} \bigg]\\
& = E\bigg[ \frac{\beta(\T^+_{k+1})(\lambda-1)(1-\beta(\T))}{(\lambda-1+\beta(\T)+\sum_{i=1}^{k+1}\beta(\T_i^+))(\lambda-1+\beta(\T)+\sum_{i=1}^{k}\beta(\T_i^+))} \bigg]\\
& \qquad + E\bigg[ \frac{-k\beta(\T^+_{k+1})\beta(\T)+ \beta(\T^+_{k+1})\sum_{i=1}^k\beta(\T_i^+)}{(\lambda-1+\beta(\T)+\sum_{i=1}^{k+1}\beta(\T_i^+))(\lambda-1+\beta(\T)+\sum_{i=1}^{k}\beta(\T_i^+))} \bigg].
\end{align*}
Since the last expectation vanishes, $B_\lambda(k+1)-B_\lambda(k)<0$ if and only if $\lambda<1$.
\end{proof}

As a consequence, when $0<\lambda\leq 1$, we have 
\begin{equation*}
\int \nu(e)\mu_{\mathsf{HARM}_\lambda}(\mathrm{d}T)> m \geq  \int \nu^+(e) \mathbf{AGW}_\lambda(\mathrm{d}T,\mathrm{d}\xi).
\end{equation*}

The next result, together with Proposition~\ref{prop:child-rw-path}, shows that $\int \nu^+(e) \mathbf{AGW}_\lambda(\mathrm{d}T,\mathrm{d}\xi)$ is not monotone with respect to $\lambda$.
\begin{proposition}
\label{prop:child-rw-path-limit-0}
As $\lambda\to 0^+$, $\int \nu^+(e) \mathbf{AGW}_\lambda(\mathrm{d}T,\mathrm{d}\xi)$ converges to $m$. 
\end{proposition}

\begin{proof}
Note that 
\begin{equation*}
\frac{\beta(\T)}{\lambda-1+\beta(\T)+\sum_{i=1}^{\nu^+}\beta(\T_i^+)} \leq 1.
\end{equation*}
By Lebesgue's dominated convergence it follows that $\lim_{\lambda\to 0^+} c_\lambda=1$.
Similarly, we have
\begin{equation*}
\lim_{\lambda\to 0^+} E\bigg[\frac{\lambda \nu^+ \beta(\T)}{\lambda-1+\beta(\T)+\sum_{i=1}^{\nu^+}\beta(\T_i^+)} \bigg]=0.
\end{equation*}
On the other hand, 
\begin{equation*}
E\bigg[\frac{(\nu^+)^2 \beta(\T)}{\lambda-1+\beta(\T)+\sum_{i=1}^{\nu^+}\beta(\T_i^+)} \bigg]=E\bigg[\frac{\nu^+\sum_{i=1}^{\nu^+}\beta(\T_i^+)}{\lambda-1+\beta(\T)+\sum_{i=1}^{\nu^+}\beta(\T_i^+)} \bigg],
\end{equation*}
to which we can apply Lebesgue's dominated convergence again to get
\begin{equation*}
\lim_{\lambda\to 0^+} E\bigg[\frac{(\nu^+)^2 \beta(\T)}{\lambda-1+\beta(\T)+\sum_{i=1}^{\nu^+}\beta(\T_i^+)} \bigg]=E[\nu^+]=m.
\end{equation*}
In view of \eqref{eq:average-nb-child-walk}, the proof is thus finished.
\end{proof}

\begin{proposition}
\label{prop:child-rw-path-limit-m}
Assume that $\sum k^3 p_k<\infty$. Then,
\begin{equation*}
\lim_{\lambda\to m-} \int \nu^+(e) \mathbf{AGW}_\lambda(\mathrm{d}T,\mathrm{d}\xi) = \frac{m^2+ \sum k^2 p_k}{2m}.
\end{equation*}
\end{proposition}

\begin{proof}
As for the analogous result in Proposition~\ref{prop:child-harm-unif},
we can use the uniform boundedness in $L^2(\mathbf{GW})$ of $\beta(T)/\mathbf{GW}[\beta(T)]$ to see that 
\begin{equation*}
\lim_{\lambda\to m^-}  c_\lambda \cdot E\bigg[ \frac{(\lambda+\nu^+) \beta(\T)}{\lambda-1} \bigg]^{-1} = 1= \lim_{\lambda\to m^-}  E\bigg[\frac{\nu^+(\lambda+\nu^+) \beta(\T)}{\lambda-1+\beta(\T)+\cc(\T^+)} \bigg] \cdot E\bigg[\frac{\nu^+(\lambda+\nu^+) \beta(\T)}{\lambda-1} \bigg]^{-1}.
\end{equation*}
Hence, it follows from 
\begin{equation*}
 E\bigg[ \frac{(\lambda+\nu^+) \beta(\T)}{\lambda-1} \bigg]^{-1} E\bigg[\frac{\nu^+(\lambda+\nu^+) \beta(\T)}{\lambda-1} \bigg]= \frac{E[\nu^+(\lambda+\nu^+)]}{E[\lambda+\nu^+]}=
 \frac{\lambda m+\sum k^2 p_k}{\lambda+m}
\end{equation*} 
that
\begin{equation*}
\lim_{\lambda\to m^-} c_\lambda^{-1} E\bigg[\frac{\nu^+(\lambda+\nu^+) \beta(\T)}{\lambda-1+\beta(\T)+\cc(\T^+)} \bigg]= \lim_{\lambda\to m^-} \frac{\lambda m+\sum k^2 p_k}{\lambda+m} 
=  \frac{m^2+ \sum k^2 p_k}{2m},
\end{equation*}
which finishes the proof by \eqref{eq:average-nb-child-walk}.
\end{proof}

Combining Propositions \ref{prop:child-harm-visi}, \ref{prop:child-harm-unif}, \ref{prop:child-rw-path-limit-0} and \ref{prop:child-rw-path-limit-m}, we see that 
\begin{equation*}
\lim_{\lambda\to 0^+} \bigg(\int \nu(e)\mu_{\mathsf{HARM}_\lambda}(\mathrm{d}T)- \int \nu^+(e) \mathbf{AGW}_\lambda(\mathrm{d}T,\mathrm{d}\xi)\bigg) = 0,
\end{equation*}
and if $\sum k^3 p_k<\infty$, 
\begin{equation*}
\lim_{\lambda\to m^-} \bigg(\int \nu(e)\mu_{\mathsf{HARM}_\lambda}(\mathrm{d}T)- \int \nu^+(e) \mathbf{AGW}_\lambda(\mathrm{d}T,\mathrm{d}\xi) \bigg)= \frac{\sum k^2 p_k -m^2}{2m} >0.
\end{equation*}
As mentioned in the introduction, we conjecture that for all $\lambda \in (0,m)$,
\begin{equation*}
\int \nu(e)\mu_{\mathsf{HARM}_\lambda}(\mathrm{d}T)- \int \nu^+(e) \mathbf{AGW}_\lambda(\mathrm{d}T,\mathrm{d}\xi)>0.
\end{equation*}

\medskip
\noindent{\bf Remark.} If we consider the average \emph{reciprocal} number of children of vertices along an infinite path in $\T$, the FKG inequality implies that for all $\lambda \in (0,m)$,
\begin{equation*}
\int \frac{1}{\nu^+(e)} \mathbf{AGW}_\lambda(\mathrm{d}T,\mathrm{d}\xi) > \frac{1}{m}. 
\end{equation*}
We also have 
\begin{equation*}
\int \frac{1}{\nu(e)} \mu_{\mathsf{HARM}_\lambda}(\mathrm{d}T) > \frac{1}{m} \mbox{ for all } \lambda \in (0,m),
\end{equation*}
by applying the FKG inequality similarly as in the proof of Proposition~\ref{prop:child-harm-unif}.

\medskip
\noindent{\bf Acknowledgment.} The author thanks Elie A\"id\'ekon and Pierre Rousselin for fruitful discussions. He is also indebted to an anonymous referee for several useful suggestions.

\end{document}